\documentclass[a4j,12pt]{amsart}
\usepackage{amsmath,amssymb,amsthm,enumerate,
}
\usepackage[all]{xy}
\usepackage{setspace}
\usepackage{indentfirst}
\usepackage{color}
\newtheorem{thm}{Theorem}[section]
\newtheorem{lemma}[thm]{Lemma}

\theoremstyle{definition}
\newtheorem{rem}[thm]{Remark}

\newtheorem{example}[thm]{Example}

\newtheorem{define-thm}[thm]{Definition-Theorem}
\usepackage{amssymb}
\usepackage[all]{xy}

\newcommand{\End}{{\rm End}}

\newcommand{\twosilt}{2\text{-}{\rm silt}\,}

\newcommand{\sttilt}{{\rm s}\tau\text{-}{\rm tilt}}
\makeatletter
\@namedef{subjclassname@2020}{%
  \textup{2020} Mathematics Subject Classification}
\makeatother
\begin{document}

\title{
On $\tau$-tilting finiteness of
block algebras of direct products of finite groups}
\author{Yuta Kozakai}
\address[Author]{Department of Mathematics,  
Tokyo University of Science}
\email{kozakai@rs.tus.ac.jp
}


\maketitle
\begin{abstract}
We discuss
finiteness/infiniteness of $\tau$-tilting modules
over tensor products of two symmetric algebras.
As an application, we consider block algebras of direct products of finite groups.
\end{abstract}

\section{Introduction}\label{introduction}
In this paper, by an algebra we mean a
finite-dimensional algebra over a fixed algebraically
closed field $k$.
In 2014, the notion of {\it support $\tau$-tilting modules}
was introduced by Adachi, Iyama and Reiten (\cite{AIR2014}),
and has been studied since by many researchers.
For example, the support $\tau$-tilting modules are
in bijection with many representation theoretical
objects such as two-term silting complexes (\cite{AIR2014}), 
functorially finite torsion classes (\cite{AIR2014}),
left finite semibricks (\cite{As2020}),
and $t$-structures (\cite{KY2014}) and more.
Since the notion of {\it $\tau$-tilting finiteness/infiniteness}
was introduced by Demonet, Iyama and Jasso \cite{DIJ2019},
several cases of $\tau$-tilting finite/infinite algebras have been verified.
For example the $\tau$-tilting finiteness/infiniteness
of the following classes of algebras are verified:
preprojective algebras of Dynkin type (\cite{M2014}),
radical square zero algebras (\cite{Ad2016}),
cycle finite algebras (\cite{MS2016}),
Brauer graph algebras (\cite{AAC2018}),
%
tame blocks of group algebras (\cite{EJR2018})
%
biserial algebras (\cite{Mo2019}),
simply connected algebras (\cite{W2019}),
tilted and cluster tilted algebras (\cite{Z2020}),
certain block algebras of finite groups covering cyclic blocks
(\cite{KK2020, KK2021, K2022}),
triangular matrix algebras (\cite{AH2021}),
certain self-injective algebras
of tubular type (\cite{AHMW2021}),
tensor product algebras of simply connected algebras (\cite{MW2021}),
classical Schur algebras (\cite{W2022})
%
and other kinds of $\tau$-tilting finite 
algebras can be seen in the list in Section 10 of \cite{T2021}.
Our aim is to give such examples among block algebras
of finite groups.
In fact, there are few studies associated to both 
$\tau$-tilting theory and modular representation theory
nevertheless the former has been developing rapidly
since it was introduced.
Therefore 
we consider the $\tau$-tilting finiteness/infiniteness
of block algebras of direct products of finite groups.
In order to do so,
we investigate
more generally the $\tau$-tilting finiteness/infiniteness of 
the tensor products of two symmetric algebras.
\begin{thm}[See Theorem \ref{theorem1}]
Let $A$ and $B$ be indecomposable symmetric algebras
over an algebraically closed field $k$
which are not local algebras.
Then the tensor product algebra $A \otimes_k B$ is a 
$\tau$-tilting infinite algebra.
\end{thm}

Finally, combining this result
and previous work on $\tau$-tilting theory,
we get the following result
on the $\tau$-tilting finiteness/infiniteness
of block algebras of direct products of finite groups.
\begin{thm}[See Theorem \ref{theorem2}]
Let $G$ and $H$ be finite groups,
$k$ an algebraically closed field
and $B$ a block of the group algebra $k[G\times H]$
of the direct product of $G$ and $H$.
If the block algebra $B$ is $\tau$-tilting finite,
then there is a block $A$ of $kG$ or of $kH$
such that the set of support $\tau$-tilting modules over $B$
is isomorphic to that of $A$ as posets.
\end{thm}

Throughout this paper,
the symbol $k$ denotes an algebraically closed field.
Algebras are always finite-dimensional over $k$
and tensor products are always taken over $k$.
For the notation of quivers and 
the basic results on $k$-linear representations of quivers,
we refer to the book \cite{ASS2006}.
In particular, for a finite dimensional algebra $A$,
we denote  by $Q_A$ the Ext-quiver of $A$
and by $\mathcal{I}_A$ the ideal of $kQ_A$
with $A\cong kQ_A/\mathcal{I}_A$.

\section{Tensor products of algebras}\label{Preliminaries}
In this section
we recall basic definitions and properties
of tensor products of algebras
and their structures of $k$-algebras.
Let $A\cong kQ_A/\mathcal{I}_A$ and $B\cong kQ_B/\mathcal{I}_B$
be algebras.
Then their tensor product $A\otimes B:=A\otimes_k B$
is again a $k$-algebra with the multiplication induced by
the rule $(a_1\otimes b_1)(a_2\otimes b_2)=a_1a_2\otimes b_1b_2$. 
To determine the structure of this $k$-algebra,
the following theorem is essential.

\begin{thm}[{\cite[Lemma 1.3]{L1994}}]
Let $Q_{A\otimes B}$ be the following quiver
$$(Q_{A\otimes B})_0=(Q_{A})_0\times (Q_{B})_0,
(Q_{A\otimes B})_1=((Q_{A})_1\times (Q_{B})_0)
\cup ((Q_{A})_0\times (Q_{B})_1).$$
Let 
$ \mathcal{I}_{A\otimes B} $
be an ideal of $kQ_{A\otimes B}$ generated by
all elements in $(\mathcal{I}_{A} \times (Q_{B})_0)
\cup ((Q_{A})_0\times \mathcal{I}_{B})$
and by all elements in
$\{(\alpha, e')(e, \beta)-(e, \beta)(\alpha, e')\:|\:
e \in (Q_A)_0, e'\in (Q_B)_0,
\alpha \in (Q_A)_1, \beta\in (Q_B)_1\}$.
Then there is a $k$-algebra isomorphism
$$
A\otimes B \cong kQ_{A\otimes B}/\mathcal{I}_{A\otimes B}.
$$

\end{thm}

\begin{example}\label{example}
Let $A$ be given by the quiver
\[
\xymatrix{
e_1 \ar@<0.5ex>[r]^{\alpha_1} 
& e_2 \ar@<0.5ex>[l]^{\alpha_2}
}
\]
bound by $\alpha^3=0$, and
$B$ be given by the quiver
\[
\xymatrix{
& e_1'\ar[ld]_{\beta_1}& \\
e_2'\ar[rr]_{\beta_2}&&e_3'\ar[lu]_{\beta_3}
}
\]
bound by $\beta^4=0$.
Then the quiver of the tensor product $A\otimes B$ 
is given by
\[
\xymatrix{
&(e_1,e_1') \ar@<0.9ex>[rrrr]|{(\alpha_1, e_1')}\ar[ldd]_{(e_1, \beta_1)}
&
&
&
&(e_2,e_1') \ar@<0.9ex>[llll]|{(\alpha_2, e_1')}\ar[ldd]_{(e_2, \beta_1)}
&
\\
&&&&&&\\
(e_1,e_2') \ar@<0.5ex>@/^60pt/[rrrr]|{(\alpha_1, e_2')}\ar[rr]_{(e_1, \beta_2)}
&
&(e_1,e_3') \ar@<0.5ex>@/_55pt/[rrrr]|{(\alpha_1, e_3')}\ar[luu]^{(e_1, \beta_3)}
&
&
(e_2,e_2') \ar@<0.5ex>@/_55pt/[llll]|{(\alpha_2, e_2')}\ar[rr]_{(e_2, \beta_2)}
&
&(e_2,e_3') \ar@<0.5ex>@/^60pt/[llll]|{(\alpha_2, e_3')}\ar[luu]_{(e_2, \beta_3)}\\
}
\]
bound by $(\alpha, e')^3=0, (e, \beta)^4=0$
and $(\alpha, e')(e, \beta)= (e, \beta)(\alpha, e')$.
\end{example}


\section{Main Theorem}\label{mainsection}

In this section
we state the main theorems stated in Section 
\ref{introduction}, and give their proofs.
%
Let $A$ be an algebra.
We recall that an $A$-module $M$ is a brick
if $\End_A(M, M)\cong k$.
Moreover we recall that the algebra $A$
is $\tau$-tilting finite if and only if there are only 
finitely many isomorphism classes of bricks over $A$
(see \cite[Theorem 1.4]{DIJ2019}).
We prepare the following lemmas.
%

\begin{lemma}\label{lemma1}
Let $A$ be an algebra and $U$ an indecomposable $A$-module.
For the indecomposable decomposition 
${\rm soc}\,U=\bigoplus_i S_i$ of ${\rm soc}\,U$, 
if all $S_i$ are non-isomorphic to each other
and each $S_i$ does not appear as a composition factor of $U/{\rm soc}\,U$,
then the $A$-module $U$ is a brick.
\end{lemma}

\begin{proof}
To prove that the $A$-module $U$ is a brick, 
assume that there is a nonzero endomorphism $f : U\rightarrow U$ 
which is not isomorphism.
Since ${\rm End}_A(U)$ is a local algebra by the indecomposability of $U$,
the endomorphism $f$ is nilpotent.
Also since the endomorphism $f$ is nonzero,
we have that $0 \neq f(U) \leq U$. 
In particular, we have that $0 \neq {\rm soc}\,f(U) \leq {\rm soc}\,U$.
Take an indecomposable summand $S$ of ${\rm soc}\,f(U)$,
then it is also an indecomposable summand of  ${\rm soc}\,U$ too.
Now, if we take a minimal submodule $V$ of $U$ with $f(V)=S$,
it holds that ${\rm top}\,V \cong S$.
Since the simple $A$-module $S$ appears as a composition factor of $U$ only one time
and it is a direct summand of ${\rm soc}\,U$,
the $A$-submodule $V$ of $U$ must be $S$.
Hence we have that $f(S)=S$, which implies that $f^n(S)=S$ for all positive integer $n$,
but this contradicts  the fact that $f$ is nilpotent. 
Therefore all nonzero endomorphisms of $U$ are isomorphisms.
\end{proof}


\begin{lemma}\label{path-lemma}
Let $A\cong kQ_A/\mathcal{I}_A$ be 
a non-local indecomposable symmetric algebra,
where $Q_A$ is an Ext-quiver of $A$ and 
$\mathcal{I}_A$ is a relation ideal.
Then there exists a path
\begin{align*}
(1\xrightarrow{\beta^{(1)}_1} 
&\cdots 
\xrightarrow{\beta^{(l_1-1)}_{1}} 1
\xrightarrow{\beta^{(l_1)}_{1}} 2
\xrightarrow{\beta^{(1)}_2} 
\cdots 
\xrightarrow{\beta^{(l_2-1)}_{2}} 2
\xrightarrow{\beta^{(l_2)}_{2}} 3
\xrightarrow{\beta^{(1)}_{3}}\\
&\cdots 
\xrightarrow{\beta^{(l_{i-1}-1)}_{i-1}} i-1
\xrightarrow{\beta^{(l_{i-1})}_{i-1}} i
\xrightarrow{\beta^{(1)}_i} 
\cdots 
\xrightarrow{\beta^{(l_i-1)}_{i}} i
\xrightarrow{\beta^{(l_i)}_{i}} i+1
\xrightarrow{\beta^{(1)}_{{i+1}}} \cdots 
\xrightarrow{\beta^{(l_{m-1})}_{m-1}} m
\xrightarrow{\beta^{(l_m)}_{m}}1)
\end{align*}
which is not a loop in $Q_A$ with
$\beta^{(1)}_1
\beta^{(2)}_1
\cdots
\beta^{(l_1)}_1
\beta_2^{(1)}
\beta_2^{(2)}
\cdots
\beta^{(l_2)}_2
\beta^{(1)}_3
\cdots
\beta^{(l_{i})}_i
\beta^{(1)}_{i+1}
\cdots
\beta_m^{(l_m)}
\not\in \mathcal{I}_A$.
\end{lemma}
\begin{proof}
Since $A$ is a non-local symmetric algebra,
for any indecomposable projective $A$-module $P$
we have that 
$P/{\rm rad}\,P\cong {\rm soc}\,P$
and that
${\rm rad}\,P/{\rm soc}\,P$ contains a simple $A$-module
distinct to ${\rm soc}\,P$.
Hence we can take a path
satisfying the following conditions:
\begin{enumerate} 
\item
it is not a loop.
\item
the source coincides with the target.
\item
the path is not in $\mathcal{I}_A$.
\end{enumerate}
Here, if there is a path with satisfying the
three properties within the path,
we retake the shorter path.
By repeating this,
we will get the desired path.
\end{proof}

\begin{rem}
Without the assumption that
the algebra $A$ is symmetric,
Lemma \ref{path-lemma}
does not hold in general.
For example, 
for the self-injective algebra
$A:=\xymatrix{
1 \ar@<0.5ex>[r]^{\alpha}&2\ar@<0.5ex>[l]^{\beta}}/
\langle \alpha\beta, \beta\alpha\rangle$
which is a not symmetric algebra,
there cannot be a path as the one stated in Lemma  \ref{path-lemma}.
\end{rem}


We are now ready to state and prove our result
on tensor products of symmetric algebras.
For an algebra $A\cong kQ_A/\mathcal{I}_A$,
we denote by ${\rm rep}_k(Q, \mathcal{I})$
the category of finite dimensional $k$-linear
representation of $Q_A$ bound by $\mathcal{I}_A$
(see \cite[Section III]{ASS2006}).

\begin{thm}\label{theorem1}
Let $A$ and $B$ be indecomposable symmetric algebras over 
an algebraically closed field $k$ which are not local algebras.
Then the tensor product algebra $A\otimes B$ is a
$\tau$-tilting infinite algebra.
\end{thm}

\begin{proof}
If there is a multiple arrow which is not a loop
in $Q_A$ or $Q_B$, then
we can easily show that 
$A\otimes B$ is a $\tau$-tilting infinite algebra
(for example, see \cite[Corollary 1.9]{DIRT} 
or
\cite[Main Theorem 2]{Ka2017})
so we may assume that there are no multiple arrows
which are not loops
in $Q_A$ and $Q_B$. 
Since the algebras $A$ and $B$ are
symmetric algebras which are not local algebras,
by Lemma \ref{path-lemma},
we can take a path
$C_A:=(1\xrightarrow{\alpha_{i_0}} 
i_1\xrightarrow{\alpha_{i_1}} 
i_2\xrightarrow{\alpha_{i_2}} 
\cdots 
\xrightarrow{\alpha_{i_{l-1}}} i_{l}=n
\xrightarrow{\alpha_{i_{l}}}1)$
in $Q_A$ satisfying the following conditions:
\begin{enumerate}
\item
$1\leq i_1\leq i_2\leq \cdots \leq i_l=n$ and $i_{c+1}-i_c$ is equal to $0$ or $1$
for $0\leq c \leq l-1$.
\item
the path is not a loop (hence $1\neq n$).
\item
$\alpha_{i_0}\alpha_{i_1}\cdots \alpha_{i_{l}} \not\in \mathcal{I}_A.$
\end{enumerate}
We can take a similar path
$C_B:=(1\xrightarrow{\beta_{j_0}} 
j_1\xrightarrow{\beta_{j_1}} 
j_2\xrightarrow{\beta_{j_2}} 
\cdots 
\xrightarrow{\beta_{j_{l'-1}}} j_{l'}=m
\xrightarrow{\beta_{j_{l'}}}1)$
in $Q_B$.
We denote a unique arrow $i\rightarrow i+1$ in $Q_A$
by $\alpha_i$
and a unique arrow $j\rightarrow j+1$ in $Q_B$
by $\beta_j$ 
for $1\leq i \leq n$ and $1\leq j \leq m$.
In these setting, we will construct a
brick $M^{(\lambda)}$ over $A\otimes B$
parameterized by the nonzero element $\lambda$ in $k$,
and show that $M^{(\lambda)}$ is not isomorphic to
$M^{(\mu)}$ if $\lambda$ is different from $\mu$.
This shows that the tensor product algebra $A\otimes B$
is a $\tau$-tilting infinite algebra by \cite[Theorem 1.4]{DIJ2019}.

First, for an arbitrary nonzero element
$\lambda$ in $k$ we define $M^{(\lambda)}$ 
by constructing a $k$-linear representation 
$M^{(\lambda)}=
(M^{(\lambda)}_{(i,j)}, \varphi^{(\lambda)}_{(\alpha, \beta)})
_{(i,j)\in (Q_{A\otimes B})_0, (\alpha, \beta) \in
(Q_{A\otimes B})_1}$
corresponding to the module $M^{(\lambda)}$
in the category ${\rm rep}_k(Q_{A\otimes B}, \mathcal{I}_{A\otimes B})$
of finite dimensional $k$-linear representations
of $Q_{A\otimes B}$ bound by $\mathcal{I}_{A\otimes B}$.
Without loss of generality, we may suppose that
$n\leq m$.
For the projective $B$-module $P(T_1)$ corresponding to 
the simple $B$-module $T_1$,
let $U$ be a minimal quotient of $P(T_1)$ such that the simple 
$B$-module $T_{m-n+2}$ appears as a composition 
factor just one time
(we remark that we can choose such $U$
because $\beta_{j_0}\beta_{j_1}\cdots\beta_{j_{l'}}\not\in \mathcal{I}_B$).
Then we remark that the socle of $U$ is isomorphic to $T_{m-n+2}$,
so the socle of $S_1\otimes U$ is 
isomorphic to $S_1\otimes T_{m-n+2}$
for the simple $A$-module $S_1$
and that
$S_1\otimes T_{m-n+2}$
appears as a composition factor
of $S_1 \otimes U$ just one time.
Hence
in the $k$-linear representation corresponding to 
$S_1 \otimes U$,
the vertex $(1, m-n+2)$ in $(Q_{A\otimes B})_0$ 
is a sink and is associated to 1-dimensional vector space $k$.
We will construct the $k$-linear representation 
$M^{(\lambda)}=
(M^{(\lambda)}_{(i,j)}, \varphi^{(\lambda)}_{(\alpha, \beta)})
_{(i,j)\in (Q_{A\otimes B})_0, (\alpha, \beta) \in
(Q_{A\otimes B})_1}$
by adding some extra nonzero spaces 
and nonzero maps
to the $k$-linear representation corresponding to 
$S_1 \otimes U$ in 
${\rm rep}_k(Q_{A\otimes B}, \mathcal{I}_{A\otimes B})$.
At first, in addition to the $k$-linear representation
of $S_1 \otimes U$,
we associate all vertices $(n-i, m-n+2+i)$ 
and $(n-i, m-n+3+i)$ in
$(Q_{A\otimes B})_0$ for $0\leq i \leq n-2$
to the 1-dimensional vector space $k$,
where the second component is 
taken modulo $m$.
Moreover we associate all arrows 
$(\alpha_{n-i}, e_{m-n+2+i})$
and
$(e_{n-i}, \beta_{m-n+2+i})$
for $0\leq i \leq n-2$
to the identity maps,
and we associate
the arrow $(\alpha_1, e_1')$ from the vertex $(1,1)$ to
$(2,1)$
to the multiplication 
by the $1\times l$ matrix
$[\lambda, 0, \cdots, 0]$, 
where $l$ means the dimension of the vector space
corresponding to the vertex $(1,1)$
and where the first component of this vector space
$k^l$ corresponds to
${\rm top}\,(S_1\otimes U)$
isomorphic to $S_1\otimes T_1$.
This finishes the construction of the
$k$-linear representation corresponding to $M^{(\lambda)}$.
We will show that the $A\otimes B$-module $M^{(\lambda)}$
is a brick,
and that it is not isomorphic to $M^{(\mu)}$
for the nonzero element in $\mu \in k$ different from $\lambda$.

By the construction of $M^{(\lambda)}$,
it is clear that the socle of $M^{(\lambda)}$ is isomorphic to
$(S_1\oplus T_{m-n+2}) \oplus (\bigoplus_{0\leq i \leq n-2} S_{n-i}\otimes T_{m-n+3+i})
=\bigoplus_{-1\leq i \leq n-2} S_{n-i}\otimes T_{m-n+3+i}$,
where the indices of 
$S_j$ are taken modulo $n$ and those of $T_j$ are
taken modulo $m$,
whose any indecomposable summand does not appear
as composition factor of $M^{(\lambda)}/{\rm soc}\,M^{(\lambda)}$.
Therefore, by Lemma \ref{lemma1} it is enough to show that 
the $A\otimes B$-module $M^{(\lambda)}$ is indecomposable
in order to prove that it is a brick.

Assume that the $A\otimes B$-module $M^{(\lambda)}$  is not indecomposable.
Then there is an endomorphism $f : M^{(\lambda)}
\rightarrow M^{(\lambda)}$ with
$M^{(\lambda)} \cong {\rm Im}\,f \oplus {\rm Ker}\,f$.
In particular, it holds that 
${\rm soc}\,M^{(\lambda)} \cong 
{\rm soc}\,{\rm Im}\,f \oplus {\rm soc}\,{\rm Ker}\,f$,
where ${\rm soc}\,{\rm Im}\,f$ and ${\rm soc}\,{\rm Ker}\,f$
are nonzero.
By the construction of $M^{(\lambda)}$,
it holds either of the following:
\begin{enumerate}[(1)]
\item
For some $-1\leq i \leq n-3$
it holds that 
$S_{n-i}\otimes T_{m-n+3+i}$
is a direct summand of ${\rm soc}\,{\rm Im}\,f$
and that
$S_{n-(i+1)}\otimes T_{m-n+3+(i+1)}$
is a direct summand of ${\rm soc}\,{\rm Ker}\,f$,
\item
For some $-1\leq i \leq n-3$
it holds that 
$S_{n-i}\otimes T_{m-n+3+i}$
is a direct summand of ${\rm soc}\,{\rm Ker}\,f$
and that
$S_{n-(i+1)}\otimes T_{m-n+3+(i+1)}$
is a direct summand of ${\rm soc}\,{\rm Im}\,f$.
\end{enumerate}
If the case (1) holds,
then
the morphism $(f_{(i,j)})_{(i,j)\in (Q_{A\otimes B})_0}$
in ${\rm rep}_k(Q_{A\otimes B}, \mathcal{I}_{A\otimes B})$
corresponding to the endomorphism $f$
satisfies the following two equations
by the commutativity of 
$(\varphi^{(\lambda)}_{(\alpha, \beta)})_{(\alpha, \beta)\in (Q_{A\otimes B})_1}$
and $(f_{(i,j)})_{(i,j)\in (Q_{A\otimes B})_0}$
(see the diagram below):
$$\begin{cases}
 f_{(n-i,\, m-n+3+i)}\circ{\rm id}
= 
{\rm id}\circ f_{(n-(i+1),\, m-n+3+i)}
\\
{\rm id}\circ f_{(n-(i+1),\, m-n+3+i)}
= f_{(n-(i+1),\, m-n+3+(i+1))}\circ{\rm id}.
\end{cases}$$
\[
\xymatrix{
&k_{(n-(i+1), m-n+3+i)}\ar@{}[lddd]|{\circlearrowright}\ar[ld]_{{\rm id}}\ar[rd]^{{\rm id}}\ar@{.>}[dd]|{f_{(n-(i+1), m-n+3+i)}}
\ar@{}[rddd]|{\circlearrowleft}&&\\
k_{(n-i, m-n+3+i)}\ar@{.>}[dd]_{f_{(n-i, m-n+3+i)}}&  & k_{(n-(i+1), m-n+3+(i+1))}\ar@{.>}[dd]^{f_{(n-(i+1), m-n+3+(i+1))}}\\
&k_{(n-(i+1), m-n+3+i)}\ar[ld]_{{\rm id}}\ar[rd]^{{\rm id}}&&\\
k_{(n-i, m-n+3+i)}&  & k_{(n-(i+1), m-n+3+(i+1))}
}
\]
This shows that the two linear maps
$f_{(n-i,\, m-n+3+i)}: k\rightarrow k$ and
$f_{(n-(i+1),\, m-n+3+(i+1))}: k\rightarrow k$
are equal, but this contradicts
the assumption (1).
Similarly we can show that the case (2) 
cannot arise.
Therefore the $A\otimes B$-module $M^{(\lambda)}$
is indecomposable.


Finally, we show that $A\otimes B$-module $M^{(\lambda)}$
is not isomorphic to
$M^{(\mu)}$ for pairwise distinct nonzero elements
$\lambda$ and $\mu$ in $k$.
Assume that there is an isomorphism 
$f: M^{(\lambda)}\rightarrow M^{(\mu)}$,
and we denote the corresponding morphism in
${\rm rep}_k(Q_{A\otimes B}, \mathcal{I}_{A\otimes B})$
by $(f_a)_{a\in (Q_{A\otimes B})_0}$.
By the constructions of $M^{(\lambda)}$ and $M^{(\mu)}$,
via a similar argument as above,
we have the following two equations.
$$\begin{cases}
 f_{(n-i,\, m-n+3+i)}\circ{\rm id}
= 
{\rm id}\circ f_{(n-(i+1),\, m-n+3+i)}
\\
{\rm id}\circ f_{(n-(i+1),\, m-n+3+i)}
= f_{(n-(i+1),\, m-n+3+(i+1))}\circ{\rm id}.
\end{cases}$$
Summarizing the above equations, we have that
$f_{(n-i,\, m-n+3+i)}=
f_{(n-(i+1),\, m-n+3+(i+1))}$
for $-1\leq i \leq n-3$.
Hence we have that
$ f_{(1,m-n+2)} = f_{(n,m-n+3)}=
f_{(n-1,m-n+4)} =\cdots =f_{(2,1)}.$
In particular, we may suppose that 
$f_{(1,m-n+2)} = f_{(2,1)} ={\rm id}$.
Moreover we regard
$f_{(1, 1)}$ as the multiplication
by an $l\times l$ matrix $\begin{bmatrix} a_{ij}\end{bmatrix}$
from the left.
Also the direct summand $S_1\otimes T_{m-n+2}$
of ${\rm soc}\,M^{(\lambda)}$ comes from
the socle of the tensor product of the simple $A$-module $S_1$
and the $B$-module $U$ which is a
quotient module of the projective $B$-module $P(T_1)$,
and $S_1 \otimes U$ has the simple top
$S_1\otimes T_1$.
Hence 
there is a path 
$$ (1,1)\xrightarrow{(e_1, \gamma_1)}
\cdots \xrightarrow{(e_1, \beta_{m-n+1})}  (1, m-n+2) $$
from the vertex $(1,1)$
to $(1, m-n+2)$ such that the composition
of the all maps on the arrows composing the path
can be considered as the multiplication by
$1\times l$ matrix
$[x ,0 ,\cdots,0 ]$ from the left for some nonzero element $x$ in $k$
(we remark that the arrow $(e_1, \gamma_1)$
is not necessarily the arrow $(e_1, \beta_1)$).
By the commutativity of 
$(\varphi^{(\lambda)}_{(\alpha, \beta)})_{(\alpha, \beta)\in 
(Q_{A\otimes B})_1}$,
$(\varphi^{(\mu)}_{(\alpha, \beta)})_{(\alpha, \beta)\in 
(Q_{A\otimes B})_1}$
and $(f_{(i,j)})_{(i,j)\in (Q_{A\otimes B})_0}$
we have that
$$\begin{cases}
f_{(1,m-n+2)} \circ\varphi_{(e_1, \beta_{m-n+1})}
\circ \cdots \circ \varphi_{(e_1, \gamma_{1})}
= \varphi_{(e_1, \beta_{m-n+1})}
\circ \cdots \circ \varphi_{(e_1, \gamma_{1})}\circ f_{(1,1)},
\\
f_{(2, 1)}\circ\varphi_{(\alpha_1, e_1')}
= \varphi_{(\alpha_1, e_1')}\circ f_{(1,1)},
\end{cases}$$
which implies that
$$\begin{cases}
[x, 0, \cdots, 0]= [x, 0, \cdots, 0][a_{ij}],
\\
[\lambda, 0, \cdots, 0] = [\mu, 0, \cdots, 0][a_{ij}].
\end{cases}$$
Therefore we have that $\lambda=\mu$.
\end{proof}

\begin{example}

Let $A$ and $B$ be the algebras defined in
Example \ref{example}.
Since the algebras $A$ and $B$ are
symmetric Nakayama algebras,
they are finite representation type.
Hence they are $\tau$-tilting finite algebras.
We show that
the tensor product $A\otimes B$ is 
a $\tau$-tilting infinite algebra.
In fact, 
starting with
the paths
$C_A:=(e_1\xrightarrow{\alpha_1}
e_2
\xrightarrow{\alpha_2}e_1)$
and
$C_B:=(e'_1\xrightarrow{\beta_1}
e'_2
\xrightarrow{\beta_2}e'_3
\xrightarrow{\beta_3}e'_1),$
for nonzero element $\lambda$ in $k$,
let $M^{(\lambda)}$ be an $A\otimes B$-module
whose $k$-linear representation is given as follows:

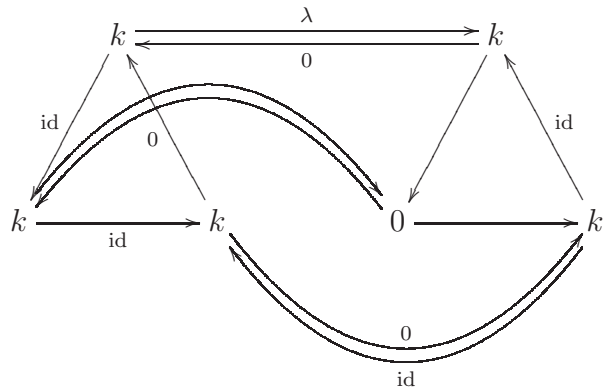
\[
\xymatrix{
&k \ar@<0.5ex>[rrrr]^{\lambda}\ar[ldd]_{{\rm id}}
&
&
&
&k \ar@<0.5ex>[llll]^{0}\ar[ldd]_{}
&
\\
&&&&&&\\
k \ar@<0.5ex>@/^50pt/[rrrr]^{}\ar[rr]_{{\rm id}}
&
&k \ar@<0.5ex>@/_50pt/[rrrr]^{0}\ar[luu]^{0}
&
&0 \ar@<0.5ex>@/_50pt/[llll]^{}\ar[rr]_{}
&
&k \ar@<0.5ex>@/^50pt/[llll]^{{\rm id}}\ar[luu]_{{\rm id}}
}
\]
Then the $A\otimes B$-module $M^{(\lambda)}$
parameterized by $\lambda$ is a brick.
Also, for a nonzero element $\mu$ in $k$ 
different from $\lambda$,
the $A\otimes B$-module $M^{(\lambda)}$ is not
isomorphic to $M^{(\mu)}$,
which means that $A\otimes B$ has
infinitely many pairwise nonisomorphic brick.
Therefore $A\otimes B$ is $\tau$-tilting infinite.

\end{example}

As an application of Theorem \ref{theorem1},
we get the following result
on the $\tau$-tilting finiteness of
block algebras of
direct products of finite groups.

\begin{lemma}[{see \cite[Theorem 2.1]{AH2021}}]\label{Lemma-AI}
Let $A$ be a finite dimensional $\tau$-tilting finite algebra
and $R$ a finite dimensional local algebra.
Then we have a poset isomorphism 
$-\otimes_k R :\sttilt A \rightarrow \sttilt(A\otimes R)$. 
\end{lemma}

\begin{proof}
It is enough to show that
$-\otimes_k R :\twosilt A \rightarrow \twosilt(A\otimes R)$
is an isomorphism of partially ordered sets.
By the similar argument as the proof of \cite[Theorem 2.1]{AH2021},
we have that the map 
$-\otimes_k R :\twosilt A \rightarrow \twosilt(A\otimes R)$
is an injection.
However, under the assumptions that
the algebra $A$ is a $\tau$-tilting finite algebra,
the similar proof to the latter half of the proof of \cite[Theorem 2.1]{AH2021} works,
and we have that the injection is a poset isomorphism.
\end{proof}

\begin{thm}\label{theorem2}
Let $G$ and $H$ be finite groups,
$k$ an algebraically closed field
and $B$ a block of the group algebra $k[G\times H]$
of the direct product of $G$ and $H$.
If the block algebra $B$ is $\tau$-tilting finite,
then there is a block $A$ of $kG$ or of $kH$
such that the set of support $\tau$-tilting modules over $B$
is isomorphic to that of $A$ as posets.
\end{thm}

\begin{proof}
Let $kG=\bigoplus_i b_i$ and $kH=\bigoplus_j b'_j$
be block decompositions.
Then the block decomposition of $k[G\times H]$
is given by $k[G\times H]=\bigoplus_{i,j} b_i \otimes b'_j$.
Since the block algebras of finite groups 
are symmetric algebras,
if both $b_i$ and $b'_j$ are not local algebras,
then the block algebra $b_i \otimes b'_j$ are
a $\tau$-tilting infinite algebra by Theorem \ref{theorem1}.
On the other hand, if the block algebra $b_i$
is a local algebra, then
the set of the support $\tau$-tilting 
modules over $b_i \otimes b'_j$ is isomorphic to 
that of $b'_j$ by 
Lemma \ref{Lemma-AI}.
The case the block algebra $b'_j$ is a local algebra
follows similarly too.
\end{proof}


\noindent
{\bf Acknowledgments}
The author would like to thank the referee for valuable comments to improve the paper.


\begin{thebibliography}{00}

\bibitem{AIR2014}
T. Adachi, O. Iyama, I. Reiten,
{\it $\tau$-tilting theory.}
 Compos. Math. {\bf 150} (2014), no. 3, 415--452.


\bibitem{Ad2016}
T. Adachi, {\it Characterizing $\tau$-tilting 
finite algebras with radical square zero.}
 Proc. Amer. Math. Soc. 144 (2016), no. 11, 4673--4685. 

\bibitem{AAC2018}
T. Adachi, T. Aihara, A. Chan,
{\it Classification of two-term tilting complexes over Brauer graph algebras.}
Math. Z. {\bf 290} (2018), no. 1--2, 1--36.

\bibitem{AH2021}
T. Aihara, T. Honma,
{$\tau$-tilting finite triangular matrix algebras.}
J. Pure Appl. Algebra {\bf 225} (2021), no. 12, 
Paper No. 106785, 10 pp. 

\bibitem{AHMW2021}
T. Aihara, T. Honma, K. Miyamoto, Q. Wang,
{\it Report on the finiteness of silting objects.}
Proc. Edinb. Math. Soc. (2) {\bf 64} (2021), no. 2, 217--233.
 
\bibitem{As2020}
S. Asai, {\it Semibricks.} 
Int. Math. Res. Not. IMRN 2020, no. 16, 4993--5054.

\bibitem{ASS2006}
I. Assem, D. Simson, A. Skowro\'{n}ski,
Elements of the representation theory of associative algebras. Vol. 1. 
Techniques of representation theory, London Mathematical Society Student Texts, 
vol. 65, Cambridge University Press, Cambridge, 2006.

\bibitem{DIRT}
L. Demonet, O. Iyama, N. Reading, I. Reiten, H. Thomas,
{\it Lattice theory of torsion classes}
Preprint (2021), arXiv:1711.01785.

\bibitem{DIJ2019}
L. Demonet, O. Iyama,
G. Jasso.
{\it $\tau$-tilting finite algebras, bricks, and $g$-vectors.} 
Int. Math. Res. Not. IMRN 2019, no. 3, 852--892.

\bibitem{EJR2018}
F. Eisele, G. Janssens, T. Raedschelders, 
{\it A reduction theorem for $\tau$-rigid modules.}
Math. Z. {\bf 290} (2018), no. 3--4, 1377--1413.

\bibitem{Ka2017}
R. Kase,
{\it From support $\tau$-tilting posets to algebras},
Preprint (2021), arXiv:1709.05049.

\bibitem{KY2014}
S. Koenig, D. Yang,
{\it Silting objects, simple-minded collections,
t-structures and co-t-structures for finite-dimensional algebras.}
 Doc. Math. {\bf 19} (2014), 403--438.

\bibitem{KK2020}
R. Koshio, Y. Kozakai,
{\it On support $\tau$-tilting modules over blocks covering cyclic blocks.}
J. Algebra {\bf 580} (2021), 84--103.

\bibitem{KK2021}
R. Koshio, Y. Kozakai,
{\it Induced modules of support $\tau$-tilting modules
and extending modules of semibricks over blocks of finite groups.}
Preprint (2021), arXiv:2112.08897.

\bibitem{K2022}
Y. Kozakai,
{\it On tilting complexes over blocks covering cyclic blocks.}
Comm. Algebra, to appear.
DOI 10.1080/00927872.2022.2162912.


\bibitem{L1994}
Z. Leszczy\'nski,
{\it On the representation type of tensor product algebras.}
 Fund. Math. {\bf144} (1994), no. 2, 143--161. 

\bibitem{MS2016}
P. Malicki, A. Skowro\'nski,
{\it Cycle-finite algebras with finitely many $\tau$-rigid 
indecomposable modules.} Comm. Algebra {\bf 44}
 (2016), no. 5, 2048--2057.

\bibitem{MW2021}
K. Miyamoto, Q. Wang.
{\it On $\tau$-tilting finiteness of
 tensor products between simply connected algebras.}
Preprint (2021), arXiv:2106.06423.

\bibitem{M2014} 
Y. Mizuno, {\it Classifying $\tau$-tilting modules
over preprojective algebras of Dynkin type.}
 Math. Z. {\bf277} (2014), no. 3--4, 665--690.

\bibitem{Mo2019}
K. Mousavand, {\it $\tau$-tilting finiteness of biserial algebras.}
Algebr. Represent. Theory (2022), to appear.

\bibitem{T2021}
H. Treffinger,
{$\tau$-tilting theory -- An introduction},
Preprint (2021), arXiv: 2106.00426.

\bibitem{Z2020}
S. Zito, 
{\it $\tau$-tilting finite cluster-tilted algebras.}
Proc. Edinb. Math. Soc. (2) {\bf63} (2020), no. 4, 950--955.

\bibitem{W2019}
Q. Wang,
{\it On $\tau$-tilting finite simply connected algebras.}
Tsukuba J. Math. {\bf 46} (2022), no. 1, 1--37.

\bibitem{W2022} 
Q. Wang, {\it On $\tau$-tilting finiteness of the 
Schur algebra.}
J. Pure Appl. Algebra {\bf 226}
(2022), no. 1, Paper No. 106818, 29 pp. 

\end{thebibliography}
\end{document}